\documentclass[a4,12pt,reqno]{amsart}
\usepackage[english]{babel}    \usepackage[latin1]{inputenc}   \usepackage[T1]{fontenc}     \usepackage[french]{minitoc}
\usepackage[nice]{nicefrac}    \usepackage{latexsym,amsfonts}  
\usepackage{graphics}    \usepackage{ulem}       \usepackage{hhline}    \usepackage{dsfont}    \usepackage{mathrsfs}
\usepackage{fancyhdr}    \usepackage{amsmath}    \usepackage{amssymb}   \usepackage{rotating}  \usepackage{fancybox}
\usepackage{color}       \usepackage{colortbl}   \usepackage{setspace}  \usepackage{enumerate} \usepackage{amsthm}
\usepackage{multicol}    
\usepackage{amsthm}    \usepackage{varioref}  \usepackage{textcomp}
\usepackage{lmodern}     \usepackage{mathpazo}   \usepackage{euscript}  \usepackage[pdftex]{hyperref}
\usepackage{palatino}
\usepackage{listings,xcolor}
\numberwithin{equation}{section}  \makeatletter\@addtoreset{equation}{section}
   \DeclareMathSymbol{\subsetneqq}{\mathbin}{AMSb}{36}
\newtheorem {theorem}{Theorem}[section]            \newtheorem {lemma}[theorem]{Lemma}
   \newtheorem {corollary}[theorem]{Corollary}     \newtheorem {remark}[theorem]{Remark}
\newtheorem {proposition}[theorem]{Proposition}       
         
\newcommand{\amnGamma}{a_{m,n}(\Gamma|\nu,\chi)}
\newcommand{\amnpqGamma}{a_{m,n}^{p,q}(\Gamma|\nu,\chi)}
\newcommand{\amnpqlambdaGamma}{a_{m,n}^{p,q}(\lambda\Gamma|\nu_{_{\lambda\Gamma}},\chi_{_{\lambda\Gamma}})}

\newcommand{\C}{\mathbb C}         \newcommand{\Z}{\mathbb Z}     \newcommand{\N}{\mathbb N} 
           \newcommand{\Ker}{K}

\newcommand{\bz}{\overline{z}}

\newcommand{\bgamma}{\overline{\gamma}}

  \newcommand{\scal}[1]{\left<#1\right>}  
\begin{document}
\title[Analytic and arithmetic properties of $ {\Ker}^\nu_{\Gamma,\chi}$]
{Analytic and arithmetic properties of the $(\Gamma,\chi)$-automorphic reproducing kernel function}
\thanks{A. El Fardi and A. Ghanmi are partially supported by the Hassan II Academy of Sciences and Technology. This work was partially supported by a grant from the Simons Foundation.}
\author[A. El Fardi, A. Ghanmi, L. Imlal, M. Souid El Ainin]{A. El Fardi, A. Ghanmi, L. Imlal, M. Souid El Ainin}
 \address{A.S.G. - L.A.M.A., Department of Mathematics, P.O. Box 1014,  Faculty of Sciences, Mohammed V University of Rabat, Morocco}
       \email{ag@fsr.ac.ma}
\date{}
\maketitle

\begin{abstract}
We consider the reproducing kernel function of the theta Bargmann-Fock Hilbert space associated to given full-rank lattice and pseudo-character, and we deal with some of its analytical and arithmetical properties. Specially, the distribution and discreteness of its zeros are examined and analytic sets inside a product of fundamental cells is characterized and shown to be finite and of cardinal less or equal to the dimension of the theta Bargmann-Fock Hilbert space.
Moreover, we obtain some remarkable lattice sums by evaluating the so-called complex Hermite-Taylor coefficients. Some of them generalize some of the arithmetic identities established by Perelomov in the framework of coherent states for the specific case of von Neumann lattice. Such complex Hermite-Taylor coefficients are nontrivial examples of the so-called lattice's functions according the Serre terminology.
 The perfect use of the basic properties of the complex Hermite polynomials is crucial in this framework.
\end{abstract}

\section{\quad Introduction}
So far, the focus of study in \cite{GI-JMP08,GI-JMP13,GHI-Rama14} has been the spectral properties of the so-called theta Bargmann-Fock Hilbert space ${\mathcal{O}}^{\nu}_{\Gamma,\chi}(\C)$ associated to given full-rank oriented lattice $\Gamma=\Z\omega_1+\Z \omega_2$ and pseudo-character $\chi$ (see also \cite{Bump02,Abdelkader1,Abdelkader2}).
 This space is defined to be the $L^2$-functional space of holomorphic $(\Gamma,\chi)$-theta functions $f$ on $\C$ of magnitude $\nu>0$, provided that the functional equation
\begin{equation}\label{EspO}
 f(z+\gamma)= \chi(\gamma) e^{\nu \left( z + \frac{\gamma} 2 \right)\bgamma }f(z)
\end{equation}
holds for all $z\in \C$ and all $\gamma\in \Gamma$. The pseudo-character nature of $\chi$, i.e.,
\begin{equation}\label{RDQ}
\chi(\gamma+\gamma')= {\chi}(\gamma){\chi}(\gamma')e^{\frac{{\nu}} 2(\gamma\overline{\gamma'} -\bgamma \gamma')}. 
\end{equation}
is required to ${\mathcal{O}}^{\nu}_{\Gamma,\chi}(\C)$ be a nonzero vector space. This is also a necessary condition and implies in particular that the imaginary part of $(\nu/\pi) \omega_1\overline{\omega_2}$ belongs to $\Z^+$. In this case,  ${\mathcal{O}}^{\nu}_{\Gamma,\chi}(\C)$ is moreover a reproducing kernel Hilbert space with respect to the positive definite hermitian product
\begin{align}
\scal{f, g}_{\nu,\Gamma}=\int_{\Lambda(\Gamma)} f(z) \overline{g(z)} e^{-\nu|z|^2}dm(z).
\label{InnerSP}\end{align}
Here $\Lambda(\Gamma)$ is any compact fundamental region for $\Gamma\backslash \C$ and $dm(z)= dxdy$; $z=x+iy$, denotes the Lebesgue measure on $\C$.
As pointed out in \cite{GI-JMP08} the reproducing kernel function is given by the absolutely and uniformly convergent series on compact subsets of $\C\times \C$,
    \begin{equation}\label{RepKer2}
 {\Ker}^\nu_{\Gamma,\chi}(z,w) := 
 \left( \frac{\nu}{\pi} \right) \sum_{\gamma\in \Gamma}\chi(\gamma)e^{-\frac \nu 2|\gamma|^2+\nu\big( z \bgamma  -\overline{w}\gamma +z \overline{w} \big)}.
\end{equation}

 Geometrically, the space ${\mathcal{O}}^{\nu}_{\Gamma,\chi}(\C)$ (under the assumption \eqref{RDQ}) can be realized as the space of holomorphic sections over the complex torus $\C/\Gamma$ of the holomorphic line bundle $L=(\C\times\C)/\Gamma$, constructed as the quotient of the trivial bundle over $\C$ by considering the $\Gamma$-action
$$
\gamma(z;v) := \left(z+\gamma; \chi(\gamma )e^{\nu \left( z + \frac{\gamma} 2 \right)\bgamma } .   v\right) ; \quad (z,v)\in \C\times \C.
$$
Accordingly, the dimension of ${\mathcal{O}}^\nu_{\Gamma,\chi}(\C)$ is known to be finite and given by the Pfaffian of the associated skew-symmetric form $E$.  In our case, $E(z,w):=(\nu/\pi)\Im m (z\bar w)$ and therefore
    \begin{equation}\label{dimFormula}
    \dim {\mathcal{O}}^{{\nu}}_{\Gamma,{\chi}}(\C)= \left(\frac\nu \pi\right) S(\Gamma),
        \end{equation}
where $S(\Gamma)$ is the cell area of $\Gamma$. Notice for instance that $\nu \geq \pi/S(\Gamma)$, otherwise ${\mathcal{O}}^{{\nu}}_{\Gamma,{\chi}}(\C)$ is trivial. The proof of \eqref{dimFormula} can be handled using Riemann-Roch theorem which is well-known in the theory of abelian varieties (see \cite{Frobenius1884,Igusa72,Swinnerton74,Lang82,Bump02,Mumford74,PolishChuk02}).
 It can also be done \`a la Selberg \cite{Selberg5789,Hejhal1998,Bump02} by determining the trace of the integral operator associated to $(\Gamma,\chi)$-automorphic kernel function ${\Ker}^\nu_{\Gamma,\chi}$  (\cite{GI-JMP08}).

In the present paper, we will concentrate on discussing some basic analytical and arithmetical properties of the reproducing kernel function ${\Ker}^\nu_{\Gamma,\chi}(z,w)$.
 More precisely, we show in Theorem \ref{Thm:M1} that the set $\mathcal{Z}({\Ker}^\nu_{\Gamma,\chi}) $ of zeros of ${\Ker}^\nu_{\Gamma,\chi}$ is symmetric, not isolated and its distribution is uniform, in the sense that $\mathcal{Z}({\Ker}^\nu_{\Gamma,\chi}) $ consists of the $\Gamma$-translation of the zeros of ${\Ker}^\nu_{\Gamma,\chi}(z,w)$ contained in a cartesian product of fundamental cells.  The systematic study of the zero set of ${\Ker}^\nu_{\Gamma,\chi}$ leads to a characterization of the common zeros (analytic set) in that cartesian product of fundamental cells, of all holomorphic functions belonging to the theta Bargmann-Fock Hilbert space ${\mathcal{O}}^{\nu}_{\Gamma,\chi}(\C)$. It is shown that this analytic set is finite and its cardinal is less or equal to $\left(\frac\nu \pi\right) S(\Gamma)$.
 In Theorem \ref{Thm:M2}, we give interesting expansion of ${\Ker}^\nu_{\Gamma,\chi}$ in power series and in terms of the so-called complex Hermite-Taylor coefficients. Moreover, we establish the connection to the Poincar\'e series associated to the monomials.
 Determination of zeros of ${\Ker}^\nu_{\Gamma,\chi}$ gives rise to interesting identities.
More remarkable lattice sums are also proved in Theorem \ref{Thm:M3} by evaluating the complex Hermite-Taylor coefficients (the coefficients of  ${\Ker}^\nu_{\Gamma,\chi}(z,w)$ when expanded as power series), generalizing certain arithmetic identities obtained by Perelomov \cite{Perelomov71} in the framework of coherent states for the specific case of von Neumann lattice \cite{vonNeumann1955} and rediscovered later by Boon and Zak \cite{BoonZak78b}.
 The proof we provide seems to be new, simpler and more direct. Theorem \ref{Thm:M4} is devoted to show that the complex Hermite-Taylor coefficients $ \amnpqGamma $ are lattice's functions in the sense of J-P. Serre \cite[Chapter VII]{Serre73}. 
\\

The rest of the paper is organized as follows.
\begin{itemize}
\item Section 2: Motivations and statement of main results.
\item Section 3: On ${\Ker}^\nu_{\Gamma,\chi}$ and proofs of Theorems \ref{Thm:M1} and \ref{Thm:M2}.
\item Section 4: Proofs of Theorems \ref{Thm:M3} and \ref{Thm:M4} related to the coefficients $ \amnpqGamma$.
\item Section 5: Concluding remarks.
\end{itemize}

\section{\quad Motivations and statement of main results}

The present work is motivated by the fact that for the von Neumann lattice \cite{vonNeumann1955}, the concrete description of ${\mathcal{O}}^{\nu}_{\Gamma,\chi}(\C)$ allows one to recover some arithmetic identities obtained by Perelomov \cite{Perelomov71} (see also \cite{BoonZak78b}).
In fact, for the specific case of $\nu=\pi/{S(\Gamma)}$ and $\chi$ being the "Weierstrass pseudo-character",
$ \chi(\gamma) = \chi_{_W}(\gamma)= +1 $ if $\gamma /2 \in \Gamma$ and  $ \chi(\gamma) = \chi_{_W}(\gamma)= -1 $ otherwise, the space ${\mathcal{O}}^{\nu}_{\Gamma,\chi}(\C)$ is one-dimensional and is generated by 
the modified Weierstrass $\sigma$-function \cite{GHI-Rama14}
\begin{equation}
\label{ModifiedSigma} \widetilde{\sigma}_\mu(z;\Gamma):= e^{-\frac 12 \mu z^2}\sigma(z;\Gamma)
= z e^{-\frac 12 \mu z^2} \prod^{\quad_{_{'}}}_{\gamma\in\Gamma} \Big(1-\frac z \gamma\Big)e^{z/\gamma + \frac 12 (z/ \gamma)^2} , 
\end{equation}
where the prime in the product excludes the term with $\gamma=0$, and $\sigma(z;\Gamma)$ denotes the classical Weierstrass $\sigma$-function
and $\mu =\mu(\Gamma)$ is an invariant of the lattice $\Gamma= \Z \omega_1 + \Z \omega_2$ given in terms of the Weierstrass zeta-function $\zeta(z;\Gamma)={\sigma^{'}(z;\Gamma)}/{\sigma(z;\Gamma)}$ 
by
 \begin{equation} \label{MuGamma}
\mu =\mu(\Gamma)=\frac{2i}S \Big({\zeta(\omega_1^*;\Gamma)\overline{\omega_2^*} -\zeta(\omega_2^*;\Gamma)\overline{\omega_1^*}}\Big) 
\end{equation}
with $\omega_\ell^*= 2\omega_\ell$; $\ell=1,2$.
In such case, up to nonzero constant $C$ the expression of ${\Ker}^\nu_{\Gamma,\chi_{_W}}(z,w)$  reduces further to
$$
{\Ker}^\nu_{\Gamma,\chi_{_W}}(z,w) = C \, e^{\frac 1 2 (\mu z^2 + \overline{\mu} \, \overline{w}^2)} \sigma(z;\Gamma) \overline{\sigma(w;\Gamma)}.
$$
Therefore, one concludes easily that the zeros of ${\Ker}^\nu_{\Gamma,\chi_{_W}}(z,w)$ are located in $(\Gamma\times \C) \cup (\C \times \Gamma)$.
As immediate consequence, one asserts that ${\Ker}^\nu_{\Gamma,\chi_{_W}}(0,0) = 0$, or equivalently
 \begin{equation} \label{PerelomocIda}
   \sum_{\gamma\in \Gamma}\chi_{_W}(\gamma)e^{-\frac \nu 2|\gamma|^2} = 0.
   \end{equation}
Thus, one recovers the arithmetic identity obtained by Perelomov in \cite[Eq. (47) with $k=0$]{Perelomov71} (see also (15) and (15a) in \cite{BoonZak78b}). 

For arbitrary lattice, the lattice sum in \eqref{PerelomocIda} as well as more arithmetic identities can arise by considering and evaluating the so-called complex Hermite-Taylor coefficients
    \begin{align}\label{aHmnRK}
\amnGamma  : = \sum_{\gamma\in \Gamma}\chi(\gamma)  e^{-\frac {\nu }2|\gamma|^2} H_{m,n}^\nu(\gamma;\bgamma ),
\end{align}
where $H_{m,n}^\nu(z,\bz )$ denote the weighted complex Hermite polynomials (see \cite{s} and the references therein)
         \begin{align}\label{chp}
       H_{m,n}^\nu(z,\bz ) :=  (-1)^{m+n}  e^{\nu |z|^2} \dfrac{\partial ^{m+n}}{\partial \bz^{m} \partial z^{n}} \left(e^{-\nu |z|^2 }\right) .
         \end{align}
This is possible thanks to the crucial observation that ${\Ker}^\nu_{\Gamma,\chi}(z,w)$ can be expanded as in \eqref{RepKer2a} below.
Thus, the present paper is aimed to investigate and complete the study of the kernel function given by both \eqref{RepKer2} and \eqref{RepKer2a} by considering further interesting basic properties of ${\Ker}^\nu_{\Gamma,\chi}(z,w)$. Special attention is given to zeros. Namely, we prove the following

\begin{theorem} \label{Thm:M1}
The set $\mathcal{Z}({\Ker}^\nu_{\Gamma,\chi}) $ of zeros of ${\Ker}^\nu_{\Gamma,\chi}$ is symmetric and is given by
$$ \mathcal{Z}({\Ker}^\nu_{\Gamma,\chi}) = \bigcup_{w\in \Lambda(\Gamma)} \bigg(\big(\mathcal{Z}(\varphi_w|_{\Lambda(\Gamma)}) + \Gamma\big) \times \{w\} \bigg) \cup  \bigcup_{w\in \Lambda(\Gamma)} \bigg( \{w\} \times (\mathcal{Z}(\varphi_w|_{\Lambda(\Gamma)}) + \Gamma) \bigg) ,$$
where $\mathcal{Z}(\varphi_w|_{\Lambda(\Gamma)})$ denotes the set of zeros of the function $\varphi_w(z):={\Ker}^\nu_{\Gamma,\chi}(z,w)$ contained in an arbitrary fundamental cell $\Lambda(\Gamma)$. Moreover, the analytic set (the common zeros) of the theta Bargmann-Fock Hilbert space $\bigcap\limits_{f\in {\mathcal{O}}^\nu_{\Gamma,\chi}(\C)} \mathcal{Z}(f|_{\Lambda(\Gamma)})$ reduces further to
$$  \Xi := \left\{ \widetilde{w} \in \Lambda(\Gamma); \, \varphi_{\widetilde{w}} \equiv 0 \mbox{ on } \Lambda(\Gamma) \right\} .$$
This set $\Xi$ is finite and its cardinal is less or equal to the dimension of ${\mathcal{O}}^\nu_{\Gamma,\chi}(\C)$.
\end{theorem}

\begin{remark}\label{Rem:M1}
The zero set $\mathcal{Z}({\Ker}^\nu_{\Gamma,\chi}|_{\Lambda(\Gamma)\times \Lambda(\Gamma)})$ is not-discrete and the distribution of zeros of ${\Ker}^\nu_{\Gamma,\chi}$ is uniform in the sense that
$$\mathcal{Z}({\Ker}^\nu_{\Gamma,\chi}) =  \mathcal{Z}({\Ker}^\nu_{\Gamma,\chi}|_{\Lambda(\Gamma)\times \Lambda(\Gamma)}) + \Gamma \times \Gamma.$$
The analytic set of the family of functions $\varphi_w|_{\Lambda(\Gamma)}$; $w\in \Lambda(\Gamma)$, coincides with $\Xi$.
\end{remark}

The second main theorem concerns the expansion in power series of ${\Ker}^\nu_{\Gamma,\chi}$ and the connection to $\mathcal{P}^\nu_{\Gamma,\chi} (e_m) $, the periodization \`a la Poincar\'e of the monomials $e_m(z) = z^m$. The $(\Gamma,\chi)$-Poincar\'e theta operator is defined by
    \begin{align}\label{PoincareOp}
 \mathcal{P}^\nu_{\Gamma,\chi} (\psi )(z)
   &: = \sum_{\gamma\in \Gamma} \chi(\gamma) e^{ -\frac{\nu} 2 |\gamma|^2 +  \nu z \bgamma } \psi (z-\gamma ),
\end{align}
provided that the series converges. More precisely, we assert

\begin{theorem} \label{Thm:M2}
The $(\Gamma,\chi)$-automorphic reproducing kernel function can be expanded in power series as follows
   \begin{align}\label{RepKer2a}
 {\Ker}^\nu_{\Gamma,\chi}(z,w)  &= \left( \frac{\nu}{\pi} \right) \sum_{m,n=0}^{\infty}  (-1)^m\amnGamma   \frac{z^n \overline{w}^m }{m!n!} ,
\end{align}
where the coefficients $\amnGamma  $ stand for
    \begin{align}\label{aHmnRK}
\amnGamma  : = \sum_{\gamma\in \Gamma}\chi(\gamma) e^{-\frac{\nu}{2} |\gamma|^2} H_{m,n}^\nu(\gamma;\bgamma ).
\end{align}
Moreover, its expression in terms of the automorphic functions obtained by periodization \`a la Poincar\'e of the monomials is given by
 \begin{align}\label{RepKer2c}
 {\Ker}^\nu_{\Gamma,\chi}(z,w)&=  \left( \frac{\nu}{\pi} \right) \sum_{m=0}^{\infty}  \left[\mathcal{P}^\nu_{\Gamma,\chi} (e_m)\right](z) \frac{(\nu \overline{w})^m }{m!}.
\end{align}
\end{theorem}

The involved $\amnGamma$, called here complex Hermite-Taylor coefficients, possess interesting arithmetic properties.
More generally, we consider the quantities
    \begin{align}\label{aHmnpqRK}
 \amnpqGamma  : = \sum_{\gamma\in \Gamma}\chi(\gamma)\gamma^p\bgamma^q e^{-\frac{\nu}{2}|\gamma|^2} H_{m,n}^\nu(\gamma;\bgamma )
\end{align}
for given integers $m,n,p,q$. Thus, we prove

 \begin{theorem} \label{Thm:M3}
   Assume that $\chi$ is real-valued. Then, the generalized complex Hermite-Taylor coefficients $\amnpqGamma $ vanish
   \begin{align}\label{amn0}
    \sum_{\gamma\in \Gamma} \chi(\gamma) \gamma^p \bgamma^q e^{-\frac{\nu}{2}|\gamma|^2} H^\nu_{m,n}(\gamma,\bgamma) = 0
    \end{align}
   for every integers $m,n,p,q$ such that $m+n+p+q$ is odd. 
\end{theorem}

\begin{remark}\label{Rem:M3}
 Under the assumption $\chi$ is a real-valued pseudo-character, we have $a_{2m+1,2n}(\Gamma|\nu,\chi)=0=a_{2m,2n+1}(\Gamma|\nu,\chi)$.
  In particular, we have the arithmetic identities
    \begin{align}\label{amn01}
    \sum_{\gamma\in \Gamma} \chi(\gamma) \gamma^{2k+1}  e^{-\frac{\nu}{2}|\gamma|^2} = 0
    \end{align}
    and
      \begin{align}\label{amn02}
    \sum_{\gamma\in \Gamma} \chi(\gamma) \bgamma^{2k+1} e^{-\frac{\nu}{2}|\gamma|^2}  = 0
    \end{align}
that follow by taking $p=q=0$ in \eqref{amn0}. These identities generalize those obtained in \cite{Perelomov71,BoonZak78b}. 
Thus, the expansion series \eqref{RepKer2a} of the reproducing kernel function ${\Ker}^\nu_{\Gamma,\chi}$ reduces further to the following
 \begin{align}\label{RepKer2b}
 {\Ker}^\nu_{\Gamma,\chi}(z,w)&= \left( \frac{\nu}{\pi} \right) \left(\sum_{m,n=0}^{\infty} a^\nu_{\Gamma,\chi}(2m,2n) \frac{z^{2n} \overline{w}^{2m} }{(2m)!(2n)!} - \sum_{m,n=0}^{\infty} a^\nu_{\Gamma,\chi}(2m+1,2n+1) \frac{z^{2n+1} \overline{w}^{2m+1} }{(2m+1)!(2n+1)!}\right).
\end{align}
\end{remark}

The lattice sums $\amnpqGamma $ will be seen as functions in the lattice $\Gamma\in \mathcal{L}$, where $\mathcal{L}$ denotes the set of all lattices in $\C$.
Hence, the last theorem is devoted to show that the complex Hermite-Taylor coefficients $\amnGamma $ are lattice's functions.

\begin{theorem} \label{Thm:M4}
There exist specific $\nu$ and $\chi$ such that $\amnpqGamma $ are lattice's functions in sense that for every nonzero complex number $\lambda$, we have
    \begin{align}\label{LatticeFct}
   \amnpqlambdaGamma     =\frac{ \lambda^{p} \overline{\lambda}^{q} }{ \overline{\lambda}^{m}\lambda^{n}  }  \amnpqGamma .
    \end{align}
\end{theorem}

\begin{remark}
The generalized complex Hermite-Taylor coefficients $\amnpqGamma$ are nontrivial examples of the so-called lattice's functions according the Serre terminology \cite[Chap VII]{Serre73}.
\end{remark}

\section{\quad On ${\Ker}^\nu_{\Gamma,\chi}$ and proofs of Theorems \ref{Thm:M1} and \ref{Thm:M2}}

The following result is easy to obtain. 

\begin{proposition} \label{prop:ExpPositivity}
The $z$-function
$$
\sum_{\gamma\in \Gamma}\chi(\gamma)e^{-\frac \nu 2|\gamma|^2+\nu\big( z \bgamma  -\overline{z}\gamma \big)} $$
is nonnegative real-valued function on the diagonal of $\C$.
\end{proposition}

\begin{proof}
This follows from the fact that
    \begin{equation}\label{RepKerzz}
 {\Ker}^\nu_{\Gamma,\chi}(z,z) =
 \left( \frac{\nu}{\pi} \right) \sum_{\gamma\in \Gamma}\chi(\gamma)e^{-\frac \nu 2|\gamma|^2+\nu\big( z \bgamma  -\overline{z}\gamma + |z|^2 \big)},
\end{equation}
and that the kernel function ${\Ker}^\nu_{\Gamma,\chi}$ is nonnegative real-valued function on the diagonal of $\C\times \C$ for satisfying
\begin{align}
{\Ker}^\nu_{\Gamma,\chi}(z,z)  = \int_{\Lambda(\Gamma)}  |{\Ker}^\nu_{\Gamma,\chi}(z,w) |^2 e^{-\nu|w|^2} dm(w) \geq 0 .
\end{align}
Indeed, making use of the reproducing property (\cite{GI-JMP08})
\begin{equation}\label{Rep2}
f(z) =  \int_{\Lambda(\Gamma)}  {\Ker}^\nu_{\Gamma,\chi}(z,w) f(w) e^{-\nu|w|^2}dm(w) ,
\end{equation}
for every $f\in {\mathcal{O}}^{\nu}_{\Gamma,\chi}(\C)$, one sees that for every fixed $\xi \in \C$ we have
$$
{\Ker}^\nu_{\Gamma,\chi}(z,\xi) =  \int_{\Lambda(\Gamma)}  {\Ker}^\nu_{\Gamma,\chi}(z,w) {\Ker}^\nu_{\Gamma,\chi}(w,\xi) e^{-\nu|w|^2}dm(w) .
$$
\end{proof}

The proof of Theorem \ref{Thm:M1} is contained in Lemmas \ref{Z1} and \ref{Z2}, Propositions \ref{Z3} and \ref{prop:Xi}, and Remarks \ref{Rem3} and \ref{setZ1} below.

  \begin{lemma}\label{Z1}
\label{prop:ExpSymmetrie}
The zeros set $\mathcal{Z}({\Ker}^\nu_{\Gamma,\chi}) $ of ${\Ker}^\nu_{\Gamma,\chi}$ is symmetric.
\end{lemma}

\begin{proof}
The symmetry of $\mathcal{Z}({\Ker}^\nu_{\Gamma,\chi}) $  follows easily since ${\Ker}^\nu_{\Gamma,\chi}(z,w) =\overline{{\Ker}^\nu_{\Gamma,\chi}(w,z)}$.
\end{proof}

  \begin{lemma} \label{Z2}
\label{prop:ExpZeros} Let $(z_0,w_0)\in \C\times \C$. Then,
${\Ker}^\nu_{\Gamma,\chi}(z_0,w_0)=0$ if and only if $ \left( \{z_0\} + \Gamma \right) \times \left( \{w_0\} + \Gamma )\right)$
is contained in the zeros set $\mathcal{Z}({\Ker}^\nu_{\Gamma,\chi}) $ of ${\Ker}^\nu_{\Gamma,\chi}$.
  \end{lemma}

\begin{proof} We need only to prove "only if". Indeed, making use of the $\Gamma$-bi-invariance property (\cite{GI-JMP08}),
  \begin{equation}\label{GbiInv}
  {\Ker}^\nu_{\Gamma,\chi}(z+\gamma,w+\gamma') = \chi(\gamma) e^{\frac{\nu} 2|\gamma|^2 + \nu z\overline{\gamma}}
     {\Ker}^\nu_{\Gamma,\chi}(z,w) \overline{\chi(\gamma')} e^{\frac{\nu} 2|\gamma'|^2 + \nu\overline{w}\gamma'}
  \end{equation}
 that holds for every $z,w\in \C$ and $\gamma, \gamma' \in \Gamma$,
it follows readily that the elements of the set $(\{z_0\} + \Gamma ) \times (  \{w_0\} + \Gamma )$ are also zeros of ${\Ker}^\nu_{\Gamma,\chi}$ whenever
${\Ker}^\nu_{\Gamma,\chi}(z_0,w_0)=0$.
\end{proof}

The following result describes the set $\mathcal{Z}(\varphi_w)$ of all zeros of $\varphi_w$ in terms of $\mathcal{Z}(\varphi_w|_{\Lambda(\Gamma)})$
that denotes the set of zeros of $\varphi_w$ inside the fundamental cell $\Lambda(\Gamma)$.

\begin{proposition} \label{Z3}
We have
$$\mathcal{Z}(\varphi_w) = \mathcal{Z}(\varphi_w|_{\Lambda(\Gamma)}) + \Gamma .$$
Moreover, the number of zeros of $\varphi_w$ contained in any fundamental cell $\Lambda(\Gamma)$ is constant and independent of $w$.
More exactly, we have
$$\#(\mathcal{Z}(\varphi_w|_{\Lambda(\Gamma)})) = \dim {\mathcal{O}}^{{\nu}}_{\Gamma,{\chi}}(\C).$$
\end{proposition}

\begin{remark}
The assertion of Proposition \ref{Z3} is valid for any $f\in {\mathcal{O}}^{\nu}_{\Gamma,\chi}(\C)$.
\end{remark}

\begin{proof}
The first assertion is clear by means of Lemma \ref{Z2}. 
 We need only to determine the number of zeros of the holomorphic function $\varphi_w:= {\Ker}^\nu_{\Gamma,\chi}(\cdot,w)$ inside $\Lambda(\Gamma)$.
 By a small displacement of $\Lambda(\Gamma)$, say $ u +  \Lambda(\Gamma)$ for certain $u\in \C$, we are free to assume that $f\ne 0$ along the border $\partial\Lambda(\Gamma)$ of a fundamental region $\Lambda(\Gamma)$.
  Without loss of generality, we can assume that $\partial\Lambda(\Gamma)$ is a piecewise smooth path that runs around each zero of $\varphi_w$ in $\Lambda(\Gamma)$ exactly one time and that the summits are the origin $O$, $\omega_1$, $\omega_1+\omega_2$ and $\omega_2$. Let denote by  $\mathcal{Z}(\varphi_w|_{\Lambda(\Gamma)})$ the set zeros of $\varphi_w$ inside such $\Lambda(\Gamma)$. hence, by applying the principle argument to $\varphi_w$, we get
  \begin{align*}
\#&(\mathcal{Z}(\varphi_w|_{\Lambda(\Gamma)})) = \frac{1}{2i\pi} \oint_{\partial \Lambda(\Gamma)} \frac{\varphi_w'(z)}{\varphi_w(z)}dz
 = \frac{1}{2i\pi} \left(\int_{O}^{\omega_1} + \int_{\omega_1}^{\omega_1+\omega_2} + \int_{\omega_1+\omega_2}^{\omega_2}   +  \int_{\omega_2}^{0} \right) \\
  &= \frac{1}{2i\pi}
  \int_{0}^{1} \left[\omega_1 \left(\frac{\varphi_w'(t\omega_1)}{\varphi_w(t\omega_1)} - \frac{\varphi_w'(t\omega_1+\omega_2)}{\varphi_w(t\omega_1+\omega_2)}\right)  + \omega_2 \left( \frac{\varphi_w'(\omega_1+t\omega_2)}{\varphi_w(\omega_1+t\omega_2)} -
   \frac{\varphi_w'(t\omega_2)}{\varphi_w(t\omega_2)}\right)\right] dt .\end{align*}
Now, according to the well established fact
 $$ \frac{f'(z+\gamma)}{f(z+\gamma)} -  \frac{f'(z)}{f(z)} = \nu \overline{\gamma} $$
 provided that $f(z+\gamma')\ne 0$; $\gamma'\in \Gamma$, and valid for every $f$ satisfying the functional equation \eqref{EspO}, we can prove the following
 \begin{align*}
\#(\mathcal{Z}(\varphi_w|_{\Lambda(\Gamma)}))
    = \frac{1}{2i\pi} \int_{0}^{1} \nu \left(\omega_1\overline{\omega_2} - \omega_2\overline{\omega_1}\right) dt
  =  \left(\frac{\nu}{\pi}\right) \Im m(\omega_1\overline{\omega_2} ).
 \end{align*}
 This completes the proof.
\end{proof}

\begin{remark}\label{Rem3} According to Proposition \ref{Z3}, one sees that $\mathcal{Z}(\varphi_w)\times \{w\}$ is contained in $\mathcal{Z}({\Ker}^\nu_{\Gamma,\chi}) $
 for every arbitrary fixed $w\in \C$. This shows in particular that  $\mathcal{Z}({\Ker}^\nu_{\Gamma,\chi}) $ is not discrete.
 \end{remark}

\begin{proposition} \label{prop:Xi}
We have $$\mathcal{Z}({\Ker}^\nu_{\Gamma,\chi})=\mathcal{Z}({\Ker}^\nu_{\Gamma,\chi}|_{\Lambda(\Gamma)\times\Lambda(\Gamma)})+\Gamma\times\Gamma$$
and the set
\begin{equation} \label{setXi}
\Xi=\{ \widetilde{w}\in\Lambda(\Gamma) \,  \mbox{ such that } \varphi_{\widetilde{w}}\equiv 0 \text{ on } \Lambda(\Gamma) \}
\end{equation}
is finite with
	\begin{equation}
	\#(\Xi ) \leqslant \dim{\mathcal{O}}^{{\nu}}_{\Gamma,{\chi}}(\C)= \left(\frac{\nu}{\pi}\right) S(\Gamma).
	\end{equation}
 Moreover, we have
	\begin{equation}\label{Xi} \Xi=\bigcap_{f\in{\mathcal{O}}^{{\nu}}_{\Gamma,{\chi}}(\C)}\mathcal{Z}(f)=\bigcap_{w\in\Lambda(\Gamma)}\mathcal{Z}(\varphi_{w})=\bigcap_{w\in\C}\mathcal{Z}(\varphi_{w}).
	\end{equation}
\end{proposition}

\begin{proof}
The distribution of zeros of ${\Ker}^\nu_{\Gamma,\chi}$ is uniform in the sense that the zeros of ${\Ker}^\nu_{\Gamma,\chi}(z,w)$ are completely determined from those located in $\Lambda(\Gamma)\times\Lambda(\Gamma)$. In fact, $(z,w)\in\mathcal{Z}({\Ker}^\nu_{\Gamma,\chi})$ is equivalent to the existence of unique $(z_0,w_0)\in \Lambda(\Gamma)\times\Lambda(\Gamma)$ and unique $(\gamma_0,\gamma_0')\in \Gamma\times\Gamma$ such that $(z,w)=(z_0,w_0)+(\gamma_0,\gamma_0')$ and $(z_0,w_0)\in\mathcal{Z}({\Ker}^\nu_{\Gamma,\chi}).$
Therefore,
 \begin{equation}
 \mathcal{Z}({\Ker}^\nu_{\Gamma,\chi})= \mathcal{Z}({\Ker}^\nu_{\Gamma,\chi}|_{\Lambda(\Gamma)\times\Lambda(\Gamma)})+\Gamma\times\Gamma.
 \end{equation}
Accordingly, we will concentrate on $\mathcal{Z}({\Ker}^\nu_{\Gamma,\chi}|_{\Lambda(\Gamma)\times\Lambda(\Gamma)})$. Recall that from Proposition \ref{Z3}, every nonzero $f$ belonging to ${\mathcal{O}}^{{\nu}}_{\Gamma,{\chi}}(\C)$, the number of zeros of $f$ in $\Lambda(\Gamma)$ is constant and given by
\begin{equation}
\#(\mathcal{Z}(f|_{\Lambda(\Gamma)})) =  \dim {\mathcal{O}}^{{\nu}}_{\Gamma,{\chi}}(\C) =\left(\frac{\nu}{\pi}\right) S(\Gamma).
\end{equation}
Thus, definition of $\Xi$ given \eqref{setXi}, one can exhibit $\widetilde{w}_0\in {\mathcal{O}}^{{\nu}}_{\Gamma,{\chi}}(\C)$ such that $\varphi_{\widetilde{w}_0}(z)=0$ for all $z$ in $\Lambda(\Gamma)$, whenever $\Xi \ne \emptyset$. Therefore,
\[
f(\widetilde{w}_0)=\int_{\Lambda(\Gamma)}{\Ker}^\nu_{\Gamma,\chi}(\widetilde{w}_0,z)f(z)e^{-\nu|z|^2}dm(z)
=\int_{\Lambda(\Gamma)}\overline{\varphi_{\widetilde{w}_0}(z)}f(z)e^{-\nu|z|^2}dm(z)
=0. 
\]
Hence $\widetilde{w}_0\in \Xi$ and then
\begin{equation}\label{inclusion}
 \Xi \subset \bigcap_{f\in{\mathcal{O}}^{{\nu}}_{\Gamma,{\chi}}(\C)}\mathcal{Z}(f)\subset \bigcap_{w\in\Lambda(\Gamma), \varphi_{w}\neq 0}\mathcal{Z}(\varphi_{w}) 
 \subset \mathcal{Z}(f)
\end{equation}
for all $w\in\Lambda(\Gamma)$ and for all $f\in{\mathcal{O}}^{{\nu}}_{\Gamma,{\chi}}(\C)$. This entails in particular that $\Xi$ is finite with
\begin{equation}\label{sharp}
 0\leqslant \#(\Xi)   \leqslant  \dim {\mathcal{O}}^{{\nu}}_{\Gamma,{\chi}}(\C) = \left(\frac{\nu}{\pi}\right) S(\Gamma) .
\end{equation}

Conversely to \eqref{inclusion}, let $z_0\in\bigcap\limits_{w\in\C}\mathcal{Z}(\varphi_{w})$, then $\varphi_{w}(z_0)=0$ for all $w$ in $\C$. Therefore, $\varphi_{z_0}(w)=0$ for all $w$ in $\Lambda(\Gamma)$. This implies that $\varphi_{z_0}\equiv0$ on $\Lambda(\Gamma)$, since $\varphi_{z_0}$ admits an infinite set of zeros in $\Lambda(\Gamma)$. This shows that $z_0\in\Xi$. Hence, we conclude that
\begin{equation*}
\Xi=\bigcap_{f\in{\mathcal{O}}^{{\nu}}_{\Gamma,{\chi}}(\C)}\mathcal{Z}(f)=\bigcap_{w\in\Lambda(\Gamma)}\mathcal{Z}(\varphi_{w}).
\end{equation*}
This completes the proof.
\end{proof}

\begin{corollary}\label{pem}
The analytic sets of $\left((\mathcal{P}^\nu_{\Gamma,\chi} (e_m)|_{\Lambda(\Gamma)}\right)_m$ and $\left(\varphi_{w}|_{\Lambda(\Gamma)}\right)_{w\in \Lambda(\Gamma)}$ coincide and are given by
\begin{equation}
\bigcap_{m\in\N}\mathcal{Z}(\mathcal{P}^\nu_{\Gamma,\chi} (e_m)|_{\Lambda(\Gamma)})=\bigcap_{w\in\Lambda(\Gamma)}\mathcal{Z}(\varphi_{w}|_{\Lambda(\Gamma)})=\Xi=\{ \widetilde{w}\in\Lambda(\Gamma); \, \varphi_{\widetilde{w}}\equiv 0\}.
\end{equation}
\end{corollary}

\begin{proof}
We have already shown that $\bigcap\limits_{w\in\C}\mathcal{Z}(\varphi_{w}|_{\Lambda(\Gamma)})\subset\mathcal{Z}(f)$, for all $f$ in ${\mathcal{O}}^{{\nu}}_{\Gamma,{\chi}}(\C)$
(see \eqref{inclusion}) and in particular for $f=\mathcal{P}^\nu_{\Gamma,\chi} (e_m)$. The converse
\begin{equation}\label{p(em)}
\bigcap_{m\in\N}\mathcal{Z}(\mathcal{P}^\nu_{\Gamma,\chi} (e_m)|_{\Lambda(\Gamma)})\subset\bigcap_{w\in\C}\mathcal{Z}(\varphi_{w}|_{\Lambda(\Gamma)})=\Xi.
\end{equation}
follows easily making use \eqref{RepKer2c} in Theorem \eqref{Thm:M2} and \eqref{Xi} in Proposition \ref{prop:Xi}. Indeed, if $z_0\in \bigcap\limits_{m\in\N}\mathcal{Z}(\mathcal{P}^\nu_{\Gamma,\chi} (e_m)|_{\Lambda(\Gamma)})$. Then, for every $w\in\C$, we have
\begin{equation}
\varphi_{w}(z_0)={\Ker}^\nu_{\Gamma,\chi}(z_0,w)=\sum\limits_{m\in\N}\mathcal{P}^\nu_{\Gamma,\chi} (e_m)(z_0)\frac{{\overline{w}}^m}{m!}=0.
\end{equation}
\end{proof}

\begin{remark}\label{setZ1}
The fact that $\Xi$ is a finite set implies that
\begin{equation}
\bigcup_{w\in\Lambda(\Gamma)\setminus \Xi} \left( \mathcal{Z}(\varphi_{w}|_{\Lambda(\Gamma)})\times\{w\}\right)
\end{equation}
is a non-countable set of zeros of ${\Ker}^\nu_{\Gamma,\chi}$.
Notice also that $\mathcal{Z}({\Ker}^\nu_{\Gamma,\chi}|_{\Lambda(\Gamma)\times\Lambda(\Gamma)})$
admits accumulation points. This follows from the compactness of $\overline{\Lambda(\Gamma)\times\Lambda(\Gamma)}$ and that the set is infinite and not-discrete for
$\#(\Xi) = \dim {\mathcal{O}}^{{\nu}}_{\Gamma,{\chi}}(\C)$.
\end{remark}

We conclude this section by presenting a proof of Theorem \ref{Thm:M2}.

\begin{proof}[Proof of Theorem \ref{Thm:M2}]
  The proof of \eqref{RepKer2a} lies essentially in the fact that in \eqref{RepKer2} we recognize the exponential generating function (\cite{s})
  \begin{align}\label{genFctgHermite}
e^{ \nu (a \gamma + b \bgamma - a b)} = \sum_{m,n=0}^{\infty} \frac{a^mb^n}{m!n!} H_{m,n}^\nu(\gamma;\bgamma)
  \end{align}
  with $a=-\overline{w}$ and $b=z$.

To establish \eqref{RepKer2c}, we recall that $\left[\mathcal{P}^\nu_{\Gamma,\chi} (e_m)\right](z)$ is given through
     \begin{align}\label{PoincareOpenPf}
 \mathcal{P}^\nu_{\Gamma,\chi} (e_m)(z)
&: = \sum_{\gamma\in \Gamma} \chi(\gamma) (z-\gamma )^m e^{ -\frac{\nu} 2 |\gamma|^2 +  \nu z \bgamma }.
\end{align}
Thus,  making use of the generating function (\cite{s})
  \begin{align}\label{genFctgHermite2}
\nu^n (\overline{\xi} - z )^n e^{ \nu \xi z }   = \sum_{m=0}^{\infty} \frac{z^m}{m!} H_{m,n}^\nu(\xi;\overline{\xi})
  \end{align}
with $\xi =\bgamma$, one obtains
       \begin{align}\label{Peren}
  \left[\mathcal{P}^\nu_{\Gamma,\chi} (e_m)\right](z)
 =   \frac{(-1)^m}{\nu^m}   \sum_{n=0}^{\infty}  \amnGamma   \frac{z^n }{n!}.
\end{align}
Therefore, the relationship \eqref{RepKer2c} follows easily from \eqref{RepKer2a} and \eqref{Peren}.
The proof of Theorem \ref{Thm:M2} is completed.
\end{proof}

  \begin{remark} \label{Rem:Peren}
  As particular case of \eqref{RepKer2c} we have
\begin{align}\label{RepKer2c0}
 {\Ker}^\nu_{\Gamma,\chi}(z,0)&= \left( \frac{\nu}{\pi} \right)  \left[\mathcal{P}^\nu_{\Gamma,\chi} (e_0)\right](z) .
\end{align} 
\end{remark}

\section{\quad Proofs of Theorems \ref{Thm:M3} and \ref{Thm:M4}}

Before giving the proof of Theorem \ref{Thm:M3} we discuss some of its immediate consequences. In fact, under the assumption that the pseudo-character is real valued, we get
    \begin{align}\label{aHmnRKodd1}  
 \sum_{\gamma\in \Gamma}\chi(\gamma)e^{-\frac {\nu }2|\gamma|^2} H_{2m+1,2n}^\nu(\gamma;\bgamma )=0
\end{align}
as well as
    \begin{align}\label{aHmnRKodd2} 
 \sum_{\gamma\in \Gamma}\chi(\gamma)e^{-\frac {\nu }2|\gamma|^2} H_{2m,2n+1}^\nu(\gamma;\bgamma )=0.
\end{align}
This follows by setting $p=q=0$ in \eqref{amn0}.
The lattice sum \eqref{aHmnRKodd2} 
can follows easily from \eqref{aHmnRKodd1} 
by changing $\gamma$ by $-\gamma$ and next taking the complex conjugate, taking into account the fact $\overline{H_{m,n}^{\nu}\left(z;\bz \right)}=H_{n,m}^{\nu}\left(z;\bz \right)$.

\begin{proof}[Proof of Theorem \ref{Thm:M3}]
Under the assumption that $\chi$ is real-valued, we have $\chi(-\gamma)=\overline{\chi(\gamma)}=\chi(\gamma)$. Therefore, the symmetry
  \begin{align}\label{Symmetrie1real}
  \amnpqGamma   = (-1)^{m+n+p+q}  \amnpqGamma
    \end{align}
  follows easily from the definition of $\amnpqGamma$ combined with the well-established fact
$$H_{m,n}^\nu(-\xi,-\overline{\xi} ) = (-1)^{m+n} H_{m,n}^\nu(\xi,\overline{\xi} ) ,$$
being indeed
\begin{align*}
 \amnpqGamma   &=  \sum_{\gamma\in \Gamma} \chi(-\gamma) (-\gamma)^q (-\bgamma)^p e^{-\frac{\nu}{2}|-\gamma|^2} H^\nu_{m,n}(-\gamma,-\bgamma) \\
 &=  (-1)^{m+n+p+q}  \amnpqGamma.
\end{align*}
This entails in particular that $\amnpqGamma  = 0$ whenever $m+n+p+q$ is odd.
\end{proof}

\begin{proof}[Proof of Theorem \ref{Thm:M4}]
 To prove \eqref{LatticeFct}, let $ \nu = \nu_{_\Gamma}$ and $\chi=\chi_{_\Gamma}$ be such that
\begin{equation}\label{NuChi}
\nu_{_{\lambda\Gamma}} = |\lambda|^{-2} \nu_{_\Gamma} \quad \mbox{ and } \quad  \chi_{_{\lambda\Gamma}}(\lambda\gamma) = \chi_{_\Gamma}(\gamma)
 \end{equation}
 for every complex $\lambda \in \C\setminus \{0\}$.  Therefore
 \begin{equation}\label{GaussianLattice}
e^{-\frac{\nu_{_{\lambda\Gamma}}}{2} |\lambda\gamma|^2} = e^{-\frac{\nu_{_\Gamma}}{2} |\gamma|^2} .
\end{equation}
Moreover, using the explicit expression of the complex Hermite polynomials (\cite{s}), to wit
  \begin{align*} 
H_{m,n}^{\nu}(\xi,\overline{\xi}) 
=   m!n! \nu^{m+n}  \sum\limits_{k=0}^{min(m,n)}\frac{(-1)^{k}  }{\nu^{k} k!}\frac{\xi^{m-k}}{(m-k)!} \frac{\overline{\xi}^{n-k}}{(n-k)!},
\end{align*}
we get 
\begin{equation}\label{HermiteLattice}
H_{m,n}^{\nu_{_{\lambda\Gamma}}}(\lambda\gamma;\overline{\lambda}\bgamma)=\frac{1}{\lambda^{n}\overline{\lambda}^{m}}H_{m,n}^{\nu_{_\Gamma}}(\gamma;\bgamma).
\end{equation}
 Hence, by inserting \eqref{NuChi}, \eqref{GaussianLattice} and \eqref{HermiteLattice} in the definition of  $\amnpqlambdaGamma$, we obtain
  \begin{align*}
   \amnpqlambdaGamma = \lambda^{p-n} \overline{\lambda}^{q-m}   \amnpqGamma .
    \end{align*}
\end{proof}

\begin{remark}
The key tool in establishing \eqref{LatticeFct} is the hypothesis \eqref{NuChi}.
The existence of such lattice's functions $\nu=\nu_{_\Gamma}$ and $\chi=\chi_{_\Gamma}$ is assured by considering for example
 $$ \nu = \nu_{_\Gamma} 
 =\frac{2i\pi(\overline{\omega_1}^{k}\omega_2^{k}-\omega_1^{k}\overline{\omega_2}^{k})}
 {\overline{\omega_1}^{k+1}\omega_2^{k+1}-\omega_1^{k+1}\overline{\omega_2}^{k+1}}$$
and the Weierstrass pseudo-character
$$\chi_{_\Gamma}(\gamma)= \left\{\begin{array}{ll}
                       +1 & \qquad \mbox{if} \quad  \gamma /2 \in \Gamma \\
                       -1 & \qquad \mbox{if} \quad  \gamma /2 \not\in \Gamma
                           \end{array}\right. .
$$
Such $\nu$ and $\chi$ are not unique.
\end{remark}

We conclude this section by proving further interesting properties satisfied by the coefficients $\amnpqGamma$.

\begin{lemma} \label{Lem;amnpq}
For any pseudo-character $\chi$ and every integers $m,n,p,q$, we have
      \begin{align}\label{Symmetrie2}
    \overline{\amnpqGamma }  = (-1)^{m+n+p+q} \amnpqGamma .
    \end{align}
Moreover, the generalized complex Hermite-Taylor coefficients $\amnpqGamma $ satisfy the following recurrence formulas
   \begin{align}\label{RecurrencFa}
   \nu a_{m,n}^{p,q+1}(\Gamma|\nu,\chi) = a_{m+1,n}^{p,q}(\Gamma|\nu,\chi) + \nu n a_{m,n-1}^{p,q}(\Gamma|\nu,\chi).
    \end{align}
    and
   \begin{align}\label{RecurrencFb}
   \nu a_{m,n}^{p+1,q}(\Gamma|\nu,\chi) = a_{m,n+1}^{p,q}(\Gamma|\nu,\chi) + \nu m a_{m-1,n}^{p,q}(\Gamma|\nu,\chi).
    \end{align}
\end{lemma}

\begin{proof}
The proof of the symmetry relationship \eqref{Symmetrie2} is similar to the one provided for \eqref{Symmetrie1real}.
We need to use the fact $\overline{\chi(\gamma)}=\chi(-\gamma)$ combined with $H_{m,n}^\nu(-\xi,-\overline{\xi} ) = (-1)^{m+n} \overline{H_{n,m}^\nu(\xi,\overline{\xi} )}.$
The recurrence formula \eqref{RecurrencFa} is an immediate consequence of the three recurrence formula
 $$H^\nu_{m+1,n}(\xi,\overline{\xi} )   = \nu \overline{\xi} H^\nu_{m,n}(\xi,\overline{\xi} )   - \nu n H^\nu_{m,n-1}(\xi,\overline{\xi} )$$
satisfied by the complex Hermite polynomials (\cite{s}). While \eqref{RecurrencFb} follows easily from \eqref{RecurrencFa} by taking the complex conjugate, taking into account \eqref{Symmetrie2}.
\end{proof}

\section{\quad Concluding remarks}

Notice finally that, for the data of any $(\Gamma,\chi,\nu)$ (not necessary satisfying the pseudo-character property \eqref{RDQ}), we can define the quantities
$\amnGamma$ and in particular
$$
a_{0,0}(\Gamma|\nu,\chi) : =  \sum_{\gamma\in \Gamma}\chi(\gamma)e^{-\frac {\nu }2|\gamma|^2} .
$$
We claim that
$$ \sum_{\gamma\in \Gamma}\chi(\gamma)e^{-\frac {\nu }2|\gamma|^2} = 0  \Longleftrightarrow \nu =\frac{\pi}{S(\Gamma)}$$
for any given real-valued $\chi$, and then this is a characterization of von Neumann lattices.
In fact, from mathematica, we assert that
$  \sum\limits_{\gamma\in \Gamma}\chi(\gamma)e^{-\frac {\nu }2|\gamma|^2} $
is nonnegative when $\nu > \pi/S(\Gamma)$, vanishes when $\nu = \pi/S(\Gamma)$ and negative otherwise.
When \eqref{RDQ} is satisfied, the space ${\mathcal{O}}^\nu_{\Gamma,\chi}(\C)$ is nontrivial and $a_{0,0}(\Gamma|\nu,\chi)$ coincides with ${\Ker}^\nu_{\Gamma,\chi}(0,0)$. Moreover, we have $\nu = k\pi/S(\Gamma)$ for $k=1,2,\cdots$.

 By numerical approach, it seems that the lattice sum
$$ \sum_{\gamma\in \Gamma}\chi(\gamma)e^{-\frac {\nu }2|\gamma|^2} =\sum_{m,n=-\infty}^{+\infty} e^{\frac{-\nu}{2} \left(m^2+n^2\right)} =: f(\nu) $$ viewed as function in $ \nu \in \mathbb{R}^{*+} $ and associated to $\Gamma=\mathbb{Z}+i\mathbb{Z} $ and $\chi\equiv 1$,
is increasing with upper bound equals to 1 and with no lower bound.
\[ \begin{array}{||c|l||}
\hline\hline
\text{values of $ \nu $} & \hfill\text{Evaluation}\hfill\\
\hline
0.001&  -911.395647437\\ 
0.01 &-100\\
0.1 &-9.999993971929989999314227390629644\\  
1 & 0\\
1.25& 0.3608381973529082048783912510840871280853\\
1.5 & 0.5852075679820198541304883766877041518512\\
1.75 & 0.7276806878628701459775102868169982969575\\
2 & 0.8196872998200458995950539646962870101812\\
3 & 0.9637440634268266347151459429784258329968\\
4 & 0.9925162797522077475130738581932461342102\\
8 & 0.9999860505819289371809872750653816298146\\
8\times 10 & 0.9999999999999999999999999999999999999999\\
\hline\hline
\end{array}\]
The above values  are calculated using the following Mathematica code
\begin{lstlisting}[language=Mathematica]
NSum[(-1)^(i+j+i*j)*Exp[(-\[Nu]*Pi)/2*(i^2+j^2)],{i,-Infinity,Infinity},{j,-Infinity,Infinity},NSumTerms->30, Method->"AlternatingSigns",WorkingPrecision->40]
\end{lstlisting}

Form the summation formulas
$$ \sum_{\gamma\in \Gamma}\chi(\gamma)e^{-\frac {\nu }2|\gamma|^2} = 0, $$
when $\nu =\frac{\pi}{S(\Gamma)}$, and
$$
 \sum_{\gamma\in \Gamma} \chi(\gamma) \gamma^p \bgamma^q e^{-\frac{\nu}{2}|\gamma|^2} H^\nu_{m,n}(\gamma,\bgamma) = 0
 $$
 for real-valued $\chi$ and integers $m,n,p,q$ such that $m+n+p+q$ is odd (see \eqref{amn0} in Theorem \ref{Thm:M3}), one can drive
 nice identities involving special functions such as Jacobi theta functions (see for example \cite{Lawden2013}).
 The first summation taken in the case $ \Gamma=\mathbb{Z}+i\mathbb{Z} $ gives rise to the following identity
\begin{equation}\label{eq:theta}
\vartheta _2(0,e^{-2 \nu })^2 - \vartheta _3(0,e^{-2 \nu })^2 - 2 \vartheta _2(0,e^{-2 \nu }) \vartheta _3(0,e^{-2 \nu }) = 0.
\end{equation}
Indeed, by splitting such summation over:  i) both $ m $ and $ n $ are odd, ii) both $ m $ and $ n $ are even and lastly, iii) $ m $ and $ n $ are of different parity, and next using the following equalities 
\begin{alignat*}{4}
\mbox{i)}  &\sum_{m=-\infty }^{\infty } \sum _{n=-\infty }^{\infty } e^{ \left(\frac{-\nu}{2} \left((2 m+1)^2+(2 n+1)^2\right)\right)} 
&&= \nu^{-\frac{1}{2}}\left( \frac{\pi }{2}\right)^{\frac{1}{2}} \vartheta _2(0,e^{-2 \nu }) \vartheta _3\left(\frac{\pi }{2},e^{-\frac{\pi ^2}{2 \nu }}\right) \\
&&&\vspace{3cm}=\vartheta _2(0,e^{-2 \nu })^2\\
\mbox{ii)}& \sum _{m=-\infty }^{\infty } \sum _{n=-\infty }^{\infty } e^{\frac{-\nu}{2} \left((2 m)^2+(2 n)^2\right)} &&= \vartheta _3(0,e^{-2 \nu })^2\\
\mbox{iii)}&\sum _{m=-\infty }^{\infty } \sum _{n=-\infty }^{\infty } e^{-\nu  \left((2 m)^2+(2 n+1)^2\right)} &&= \vartheta _2(0,e^{-2 \nu }) \vartheta _3(0,e^{-2 \nu })
\end{alignat*}

The proof of the identity \eqref{eq:theta} requires detailed knowledge of the relationships
between theta functions. The proof of the more complicated identities of the type (2.11) by direct methods is evidently a nontrivial problem. 
This is, word by word, the same remark asserted by Perelomov in \cite{Perelomov71}.

\end{document}